\documentclass[11pt]{amsart}
\usepackage{amsmath, amsfonts,amsthm,amssymb,amscd, verbatim, graphicx,color,multirow,booktabs,tikz,adjustbox}
\usepackage{chngcntr}
\counterwithin{table}{section}
\def\classification#1{\def\@class{#1}}
\classification{\null}
\usepackage[margin=1.5cm]{geometry}
\usepackage{xcolor,tkz-berge}

\usepackage{hyperref}

\newcommand{\Alt}{\mathop{\mathrm{Alt}}}
\newcommand{\Sym}{\mathop{\mathrm{Sym}}}
\newcommand{\suzuki}{\mathop{{^2\mathrm{B}_2}(q)}}
\newcommand{\ree}{\mathop{{^2\mathrm{G}_2}(q)}}

\def\nor#1#2{{\bf N}_{{#1}}{{(#2)}}}
\def\Zent#1{{\bf Z}{{(#1)}}}
\def\cent#1#2{{\bf C}_{{#1}}{{(#2)}}}
\DeclareFontFamily{OT1}{rsfs}{}
\DeclareFontShape{OT1}{rsfs}{n}{it}{<-> rsfs10}{}
\DeclareMathAlphabet{\mathscr}{OT1}{rsfs}{n}{it}
 
\newcommand{\Fix}{{\rm Fix}}
\newcommand{\Fq}{\mathbb{F}_q}

\newcommand{\SL}{\mathrm{SL}}

\newcommand{\PSL}{\mathrm{PSL}}
\newcommand{\PSU}{\mathrm{PSU}}
\newcommand{\PGU}{\mathrm{PGU}}

\newcommand{\PGL}{\mathrm{PGL}}
\newcommand{\PGammaL}{\mathrm{P\Gamma L}}
\newcommand{\GL}{\mathrm{GL}}

\newcommand{\symme}{\mathrm{Sym}}
\newcommand{\alter}{\mathrm{Alt}}

\newtheorem{prop}{Proposition}[section]
\newtheorem{thm}[prop]{Theorem}
\newtheorem{conj}[prop]{Conjecture}

\newtheorem{example}[prop]{Example}

\newtheorem{lem}[prop]{Lemma}

\theoremstyle{definition}

\numberwithin{equation}{section}
\newcommand{\SLq}{{\rm PSL}_2(q)}

\theoremstyle{definition}
\begin{document}
\title{Cherlin's conjecture for almost simple groups of Lie rank $1$}  

\author{Nick Gill}
\address{ Department of Mathematics, University of South Wales, Treforest, CF37 1DL, U.K.}
\email{nick.gill@southwales.ac.uk}

\author{Francis Hunt}
\address{ Department of Mathematics, University of South Wales, Treforest, CF37 1DL, U.K.}
\email{francis.hunt@southwales.ac.uk}

\author{Pablo Spiga}
\address{Dipartimento di Matematica e Applicazioni, University of Milano-Bicocca, Via Cozzi 55, 20125 Milano, Italy} 
\email{pablo.spiga@unimib.it}

\begin{abstract}
 We prove Cherlin's conjecture, concerning binary primitive permutation groups, for those groups with socle isomorphic to $\SLq$, $\suzuki$, $\ree$ or $\PSU_3(q)$. Our method uses the notion of a ``strongly non-binary action''. 
\end{abstract}

\maketitle 

\section{Introduction}

All groups in this paper are finite. In this note our main result is the following.

\begin{thm}\label{t: psl2}
 Let $G$ be an almost simple primitive permutation group on the set $\Omega$ with socle isomorphic to a linear group $\SLq$, or to a Suzuki group $\suzuki$, or to a Ree group $\ree$, or to a unitary group $\PSU_3(q)$. Then, either $G$  is not binary, or $G=\Sym(\Omega)\cong \Sym(5)\cong \PGammaL_2(4)\cong \PGL_2(5)$, or $G=\Sym(\Omega)\cong \Sym(6)\cong \mathrm{P}\Sigma\mathrm{L}_2(9)$.
\end{thm}

%Note that, if $q\leq 3$, then $\PSL_2(q)$ is not simple. 
Theorem~\ref{t: psl2} is a contribution towards a proof of a conjecture of Cherlin \cite{cherlin1}. This conjecture asserts that a primitive binary permutation group lies on a short explicit list of known actions.  

The precise definition of ``binary'' and ``binary action'' is given in Section~\ref{s: bin back} below. An equivalent definition, couched in terms of ``relational structures'', can be found in \cite{cherlin2}; the connection between this conjecture and Lachlan's theory of sporadic structures can be found in \cite{cherlin1}. It is this connection that really enlivens the study of binary permutation groups, and provides motivation to work towards a proof of Cherlin's conjecture. 

Let us briefly describe the status of this conjecture. By work of Cherlin \cite{cherlin2} and Wiscons \cite{wiscons}, this very general conjecture has been reduced to the following statement concerning almost simple groups.

\begin{conj}\label{conj: cherlin}
 If $G$ is a binary almost simple primitive permutation group on the set $\Omega$, then $G=\symme(\Omega)$. 
\end{conj}

One sees immediately that Theorem~\ref{t: psl2} settles Conjecture~\ref{conj: cherlin} for almost simple primitive permutation groups with socle isomorphic to $\SLq$, or $\suzuki$, or $\ree$, or $\PSU_3(q)$, that is, for each Lie type group of twisted Lie rank $1$. Theorem~\ref{t: psl2} is the third recent result of this type; in recent work, the first and third authors settled Conjecture~\ref{conj: cherlin} for groups with alternating socle, and for the $\mathcal{C}_1$ primitive actions of groups with classical socle \cite{gs_binary}. 

A brief word about our methods: the aforementioned work on groups with alternating or classical socle was based on the study of so-called ``beautiful subsets''. These objects are defined below, and their usefulness is explained by Lemma~\ref{l: forbidden} and Example~\ref{ex: snba1} below, which together imply that whenever an action admits a beautiful subset the action is not binary.

In the current note our approach is different for the reason that the family of actions under consideration -- the primitive actions of almost simple groups with socle a Lie group of Lie rank $1$-- very often do not have beautiful subsets.

To deal with this situation we need to develop a more general theory: Suppose that we have a group $G$ acting on a set $\Omega$, and we want to show that this action is non-binary. The key property of beautiful subsets that makes them useful is that they allow us to argue ``inductively'', in the sense that if we can find a subset $\Lambda$ of $\Omega$ that is ``beautiful'', then the full action of $G$ on $\Omega$ is non-binary. In order to deal with the absence of beautiful subsets, we have studied this inductive property more formally via the notion of a ``strongly non-binary subset''. The theory of such subsets is developed in \S\ref{s: bin back} and allows us to apply an inductive argument in a more general setting. %We think of the action of $G^\Lambda$ on $\Lambda$ in this case as being ``strongly non-binary'' -- the adverb ``strongly'' being used to indicate that we have this inductive property. A precise description of this notion is given below in \S\ref{s: bin back}.

The advantages of Theorem~\ref{t: psl2} and of this theory are several: firstly, Theorem~\ref{t: psl2} is a material advance towards a proof of Conjecture~\ref{conj: cherlin}; secondly, it demonstrates the possibility of obtaining results in situations where one cannot use the notion of a beautiful subset, as in~\cite{gs_binary}; thirdly, it turns out that the rank $1$ groups tend to be a sticking point when making general arguments concerning binary groups. We hope, therefore, that by disposing of this case here, we will be able to deal more easily with the remaining cases required for a proof of Cherlin's conjecture. Investigation in this direction is in progress, see~\cite{gls_binary}.%This last situation seems to be particularly relevant to groups of Lie type of low rank. To demonstrate this further we include two relevant results:

%\begin{thm}\label{t: c2}
%Let $G$ be an almost simple primitive permutation group on the set $\Omega$ with socle isomorphic to $\PSL_n(q)$. If the stabilizer of a point in $\Omega$ is in the Aschbacher $\mathcal{C}_2$-class, then $G$ is not binary. 
%\end{thm}

%This result is of a type with \cite[Theorem B]{gs_binary} in that it resolves Cherlin's Conjecture for one of the families of primitive actions of $\PSL_n(q)$, as given by Aschbacher's classification of maximal subgroups \cite{aschbacher2}, as described in~\cite{kl}. The Aschbacher $\mathcal{C}_2$-class in $\PSL_2(q)$ corresponds to the normalizer of a split torus in $\PSL_2(q)$, it is important to note that very often  this action of $\PSL_2(q)$ does not admit beautiful subsets: hence, a genuinely different approach from~\cite{gs_binary} is needed to prove Theorem~\ref{t: c2}. This approach is developed in this paper.

%Similar small rank exceptions occur, for example, in the $\mathcal{C}_5$ and $\mathcal{C}_8$ families in $\PSL_n(q)$: one obtains a beautiful subset (and, thereby, a proof of Cherlin's Conjecture for these actions) once the rank is large enough; when $n$ is small, one needs a different approach \cite{gls_binary}.

%\begin{thm}\label{t: suzuki}
%Let $G$ be an almost simple primitive permutation group with socle isomorphic to a Suzuki group $\suzuki$, to a Ree group $\ree$ or to a %unitary group $\PSU_3(q)$. Then $G$ is not binary.
%\end{thm}

\subsection{Structure of the paper}

The proof of Theorem~\ref{t: psl2} is split into several parts. First, in \S\ref{s: bin back}, after giving a number of definitions, we prove some general results about binary actions; in particular Lemma~\ref{l: forbidden} is vital.

In \S\ref{s: structure} we give some basic information concerning groups with socle isomorphic to $\PSL_2(q)$; then in \S\ref{s: fp} we calculate the size of the fixed set for various elements of $\PGammaL_2(q)$ in various primitive actions; these results are then used to prove Lemmas~\ref{l: handy q odd} and \ref{l: handy q even}; it is worth remarking that these fixed point calculations yield the required conclusions almost immediately for the groups $\SLq$ and $\PGL_2(q)$, however a finer analysis is required to deal with those almost simple groups that contain field automorphisms.
The three lemmas just mentioned --  Lemmas~\ref{l: forbidden}, \ref{l: handy q odd} and \ref{l: handy q even} -- directly imply Theorem~\ref{t: psl2} for $\PSL_2(q)$ when $q\geq 9$. The remaining small cases, when $q\in\{4,5,7,8\}$, can be verified directly using GAP \cite{GAP} or by referencing the calculations of Wiscons \cite{wiscons2}.

In \S\ref{s: suzuki}, \S\ref{s: ree} and \S\ref{s: psu}, we give a proof of Theorem~\ref{t: psl2} for groups with socle $\suzuki$, $\ree$ and $\PSU_3(q)$, respectively. In the first two cases the theorems are easy consequences of propositions asserting that the primitive actions in question admit strongly non-binary subsets (see \S\ref{s: bin back} for the definition of a strongly non-binary subset). The final case -- socle $\PSU_3(q)$ -- is dealt with somewhat differently.

\section{Binary actions and strongly non-binary actions}\label{s: bin back}

Throughout this section $G$ is a finite group acting (not necessarily faithfully) on a set $\Omega$ of cardinality $t$. Here, our job is to give a definition of ``binary action'', and of ``strongly non-binary action'', and to connect these definitions to earlier work on ``beautiful sets''. Given a subset $\Lambda$ of $\Omega$, we write $G_\Lambda:=\{g\in G\mid \lambda^g\in\Lambda,\forall \lambda\in \Lambda\}$ for the set-wise stabilizer of $\Lambda$, $G_{(\Lambda)}:=\{g\in G\mid \lambda^g=\lambda, \forall\lambda\in \Lambda\}$ for the point-wise stabilizer of $\Lambda$, and $G^\Lambda$ for the permutation group induced on $\Lambda$ by the action of $G_\Lambda$. In particular, $G^\Lambda\cong G_\Lambda/G_{(\Lambda)}$.

Given a positive integer $r$, the group $G$ is called \textit{$r$-subtuple complete} with respect to the
pair of $n$-tuples $I, J \in \Omega^n$, if it contains elements that
map every subtuple of size $r$ in $I$ to the corresponding subtuple in
$J$ i.e. $$\textrm{for every } k_1, k_2, \dots, k_r\in\{ 1,
\ldots, n\}, \textrm{ there exists } h \in G \textrm{ with }I_{k_i}^h=J_{k_i}, \textrm{ for every }i \in\{
1, \ldots, r\}.$$ Here $I_k$ denotes the $k^{\text{th}}$ element of tuple
$I$ and $I^g$ denotes the image of $I$ under the action of $g$.
Note that $n$-subtuple completeness simply requires the existence of
an element of $G$ mapping $I$ to $J$.

The group $G$ is said to be of {\it arity $r$} if, for all
$n\in\mathbb{N}$ with $n\geq r$ and for all $n$-tuples $I, J \in \Omega^n$, $r$-subtuple
completeness (with respect to $I$ and $J$) implies $n$-subtuple completeness (with respect to $I$ and $J$). When $G$ has arity 2, we say that $G$ is {\it binary}. A pair $(I,J)$ of $n$-tuples of $\Omega$ is called a {\it non-binary witness for the action of $G$ on $\Omega$}, if $G$ is $2$-subtuple complete with respect to $I$ and $J$, but not $n$-subtuple complete, that is, $I$ and $J$ are not $G$-conjugate.
To show that the action of $G$ on $\Omega$ is non-binary it is sufficient to find a non-binary witness $(I,J)$. 

We say that the action of $G$ on $\Omega$ is \emph{strongly non-binary} if there exists a non-binary witness $(I,J)$ such that
\begin{itemize}
 \item $I$ and $J$ are $t$-tuples where $|\Omega|=t$;
 \item the entries of $I$ (resp. $J$) are distinct entries of $\Omega$.
\end{itemize}

\begin{example}\label{ex: snba1}{\rm
If $G$ acts $2$-transitively on $\Omega$ with kernel $K$ and $G/K\cong G^\Omega\not\cong\Sym(\Omega)$, then $G$ is strongly non-binary.
 
Indeed, by $2$-transitivity, any pair $(I,J)$ of $t$-tuples of distinct elements from $\Omega$ is $2$-subtuple complete. Since $G/K\cong G^\Omega\not\cong\Sym(\Omega)$, we can choose $I$ and $J$ in distinct  $G$-orbits. Thus $(I,J)$ is a non-binary witness.}
\end{example}

\begin{example}\label{ex: snba2}{\rm
Let $G$ be a subgroup of  $\Sym(\Omega)$, let $g_1, g_2,\ldots,g_r$ be elements of $G$, and let $\tau,\eta_1,\ldots,\eta_r$ be elements of $\Sym(\Omega)$ with
\[
 g_1=\tau\eta_1,\,\,g_2=\tau\eta_2,\,\,\ldots,\,\,g_r=\tau\eta_r.
\]
Suppose that, for every $i\in \{1,\ldots,r\}$, the support of $\tau$ is disjoint from the support of $\eta_i$; moreover, suppose  that, for each $\omega\in\Omega$, there exists $i\in\{1,\ldots,r\}$ (which may depend upon $\omega$) with $\omega^{\eta_i}=\omega$. Suppose, in addition, $\tau\notin G$.
Now, writing $\Omega=\{\omega_1,\dots, \omega_t\}$, observe that
 \[
  ((\omega_1,\omega_2,\dots, \omega_t), (\omega_1^{\tau},\omega_2^{\tau}, \ldots,\omega_t^{\tau}))
 \]
is a non-binary witness. Thus the action of $G$ on $\Omega$ is strongly non-binary.}
\end{example}

The notion of a strongly non-binary action allows us to ``argue inductively'' using suitably chosen set-stabilizers. The following lemma (which was first stated in~\cite{gs_binary} and which, in any case, is virtually self-evident) clarifies what we mean by this.

\begin{lem}\label{l: forbidden}
Suppose that there exists a subset $\Lambda \subseteq \Omega$ such that $G^\Lambda$ is strongly non-binary. Then $G$ is not binary.
\end{lem}

In what follows a \emph{strongly non-binary subset} is a subset $\Lambda$ of $\Omega$ such that $G^\Lambda$ is strongly non-binary. 

We are ready for the third main concept of this section, that of a ``beautiful subset''; this is closely related to the example of a strongly non-binary action given in Example~\ref{ex: snba1}. Specifically, we say that a subset $\Lambda\subseteq \Omega$ is a \emph{$G$-beautiful subset} if $G^\Lambda$ is a $2$-transitive subgroup of $\symme(\Lambda)$ which is neither $\alter(\Lambda)$ nor $\symme(\Lambda)$. Note that we will tend to drop the ``$G$'' in $G$-beautiful, so long as the context is clear (for instance, when $G$ is a permutation group on $\Omega$, that is, $G\le\Sym(\Omega)$).

In the light of Example~\ref{ex: snba1}, the curious reader may be wondering why the definition of a beautiful subset excludes also the possibility that $G^\Lambda=\alter(\Lambda)$. This exclusion is explained by the following lemma, which is~\cite[Corollary 2.3]{gs_binary}.

\begin{lem}\label{l: beautiful}
 Suppose that $G$ is almost simple with socle $S$. If $\Omega$ contains an $S$-beautiful subset, then $G$ is not binary.
\end{lem}

In what follows we will see a number of examples of strongly non-binary actions of the types given in Examples~\ref{ex: snba1} and \ref{ex: snba2}, as well as examples of beautiful subsets. To study these examples we will make use of the fact that the finite faithful 2-transitive actions are all known thanks to the Classification of Finite Simple Groups. %Note too that, writing $a=(1,2)(3,4)$, $b=(1,2)(5,6)$ and $c=(1,2)$, we see that the group $\langle a,b\rangle$ is a Klein-4 group, while the group $\langle a,b,c\rangle$ is elementary abelian of order $8$.

One naturally wonders whether other examples of strongly non-binary witnesses exist. This is indeed the case and the existence of a strongly non-binary witness is related to the classic concept of $2$-closure introduced by Wielandt~\cite{Wielandt}. Given a permutation group $G$ on $\Omega$, the \emph{$2$-closure of $G$} is the set $$G^{(2)}:=\{\sigma\in \Sym(\Omega)\mid \forall (\omega_1,\omega_2)\in \Omega\times \Omega, \textrm{there exists }g_{\omega_1\omega_2}\in G \textrm{ with }\omega_1^\sigma=\omega_1^{g_{\omega_1\omega_2}}, \omega_2^\sigma=\omega_2^{g_{\omega_1\omega_2}}\},$$
that is, $G^{(2)}$ is the largest subgroup of $\Sym(\Omega)$ having the same orbitals as $G$. The group $G$ is said to be $2$-closed if and only if $G=G^{(2)}$. We claim that $G$ is not $2$-closed if and only if $G$ has a strongly non-binary witness. Write $\Omega:=\{\omega_1,\ldots,\omega_t\}$. If $G$ is not $2$-closed, then there exists $\sigma\in G^{(2)}\setminus G$. Now, it is easy to verify that $I:=(\omega_1,\ldots,\omega_t)$ and $J:=I^\sigma=(\omega_1^\sigma,\ldots,\omega_t^\sigma)$ are $2$-subtuple complete (because $\sigma\in G^{(2)}$) and are not $G$-conjugate (because $g\notin G$). Thus $(I,J)$ is a strongly non-binary witness. The converse is similar.

%Example~\ref{ex: snba2} is easily generalized to the situation where $G$ contains elements $\tau_1\tau_2$ and $\tau_1\tau_3$ but not $\tau_1$ -- here $\tau_1,\tau_2,\tau_3$ are permutations of $\Omega$ with disjoint support.

%\begin{question}{\rm 
%Besides the strongly non-binary witnesses described in Examples~\ref{ex: snba1} and~\ref{ex: snba2}, is it possible to find other types? Is a full classification possible?}
%\end{question}

\section{Groups with socle isomorphic to \texorpdfstring{$\PSL_2(q)$}{PSL2(q)}}\label{s: structure}

In this section we start by  studying some of the basic properties of involutions and Klein $4$-subgroups of the almost simple groups $G$ with socle $\PSL_2(q)$. (In particular, $\PSL_2(q)\le G\le \PGammaL_2(q)$.) All of these properties are well-known and/or easy to verify by direct calculation. We also set up some basic notation for what follows.

For a group $J$, write $m_2(J)$ for the \emph{$2$-rank} of $J$, i.e.\ the maximum rank of an abelian $2$-subgroup of $J$. If $q$ is odd and $J$ is a section of $G$ (i.e.\ a quotient of a subgroup of $G$), then $m_2(J)\leq 3$. What is more, $m_2(J)\leq 2$ unless $q$ is a square and $G$ contains a field automorphism of order $2$.

%The following lemma is an immediate consequence of the subgroup structure of $\PGL_2(q)$.

\begin{lem}\label{l: quotients split}
Let $L$ be a subgroup of $\PGL_2(q)$ with $q$ odd, and let $K$ be a subgroup of $\nor {\PGL_2(q)} L $ with $K$ isomorphic to a Klein $4$-group and with $K\cap L=1$.
Then $|L|$ is odd.
\end{lem}
\begin{proof}
Let $P$ be a Sylow $2$-subgroup of $\langle K,L\rangle=K\ltimes L$ containing $K$. Then $P=K\ltimes Q$, for some Sylow $2$-subgroup $Q$ of $L$. If $Q\ne 1$, then $K$ centralises a non-identity element of $Q$ and hence $m_2(\PGL_2(q))\ge m_2(P)=m_2(K\ltimes Q)\ge m_2(K)+m_2(\cent Q K)\ge 2+1=3$, a contradiction.
\end{proof}
%We will apply Lemma~\ref{l: quotients split} to the action of $G$ on $\Omega$. We will display subsets $\Lambda \subset \Omega$ such that $G_\Lambda$, the setwise stabilizer of $\Lambda$ contains a Klein 4-subgroup $K$, none of whose elements fix all the elements of $\Lambda$. Then Lemma~\ref{l: quotients split} implies that the pointwise stabilizer of $\Lambda$ in $\PGL_2(q)$ has odd order.

%From here on, $g$ and $h$ will be involutions in $S$ such that $K=\langle g,h\rangle$ is a Klein 4-group.
Suppose that $q$ is odd. There is exactly one $\PGL_2(q)$-conjugacy class of Klein $4$-subgroups of $\PSL_2(q)$, and one can check directly that $\cent{\PGL_2(q)} K =K$ for each Klein $4$-subgroup of $\PSL_2(q)$. When $q\equiv \pm 3\pmod 8$, a Sylow $2$-subgroup of $\PSL_2(q)$ is a Klein $4$-subgroup and, by Sylow's theorems, there is exactly one $\PSL_2(q)$-conjugacy class of Klein $4$-subgroups of $\PSL_2(q)$; in this case $\nor {\PSL_2(q)} K \cong \Alt(4)$. When $q\equiv \pm 1\pmod 8$, there are two $\PSL_2(q)$-conjugacy classes of Klein $4$-subgroups of $\PSL_2(q)$ and these are fused in $\PGL_2(q)$; in this case $\nor {\PSL_2(q)} K \cong\Sym(4)$.

We need information concerning involutions in $\PGammaL_2(q)\setminus\PGL_2(q)$ -- such involutions must be field automorphisms, as defined in~\cite{gls3}. The following result is a special case of \cite[Prop. 4.9.1]{gls3}.

\begin{lem}\label{l: fields}
 Let $f_1,f_2\in\PGammaL_2(q)\setminus\PGL_2(q)$ be of order $t$ for some prime $t$, and suppose that $f_1\PGL_2(q)=f_2\PGL_2(q)$. Then $f_1$ and $f_2$ are $\PGL_2(q)$-conjugate.
\end{lem}

\subsection{Fixed point calculations}\label{s: fp}

We let $G$ be a group having socle $S$ with $S\cong \PSL_2(q)$. Using the classification of the maximal subgroups of $G$ (see for example~\cite{bhr}), it is important to observe that, for every maximal subgroup $M$ of $G$ there exists a maximal subgroup $H$ of $S$ with $M=\nor G H$; in particular, this allows us to identify (up to permutation isomorphism) each primitive $G$-set $\Omega$ with the set of $G$-conjugates of some maximal subgroup $H$ of $S$.
Therefore, we let $H$ be a maximal subgroup of $S$ with $\nor G H $ maximal in $G$, and set $\Omega$ to be $H^G:=\{H^g\mid g\in G\}$, the set of all conjugates of $H$ in $G$. All possibilities for $H$ and $|\Omega|$ are given in the first and in the third column of Tables~\ref{t: inv q odd} and \ref{t: inv q even}, where in Table~\ref{t: inv q odd} the symbol $\zeta$ is defined by
\begin{equation}\label{e: zeta}
 \zeta:=\begin{cases}
        2 & \textrm{if }G\not\le \mathrm{P}\Sigma L_2(q) \textrm{ and }q \textrm{ is odd}, \textrm{ or }q\textrm{ is even},\\
        1 & \textrm{if }G\le \mathrm{P}\Sigma L_2(q) \textrm{ and }q \textrm{ is odd}.
       \end{cases}
\end{equation}
(See \cite{bhr} to verify this. The conditions that are listed in Table~\ref{t: inv q odd} are necessary for the  action of $G$ on $\Omega$ to be primitive, but they are not necessarily sufficient.)  Finally, we write $\mathcal{P}(H)$ for the power set of $H$.

In what follows, we calculate the number of fixed points of an involution $g\in S$, and (when $q$ is odd) of a Klein $4$-subgroup $K\leq S$, for the action of $G$ on $\Omega$. (Given a subset $Y$ of a permutation group $X$ on $\Omega$, we write $\Fix_\Omega(Y):=\{\omega\in \Omega\mid \omega^y=\omega,\forall y\in Y\}$ and simply $\Fix_\Omega(y)$ when the set $Y$ consists of the single element $y$.)

To calculate the number of fixed points of $g$ and of $K$, we make use of the well-known formulas (see for instance~\cite[Lemma~$2.5$]{LiebeckSaxl})
\begin{equation}\label{e: fora}
 |\Fix_\Omega(g)| = \frac{|\Omega|\cdot |H\cap g^G|}{|g^G|},\qquad
 |\Fix_\Omega(K)| = \frac{|\Omega|\cdot |\mathcal{P}(H)\cap K^G|}{|K^G|}.
\end{equation}

Given an involution $g\in S$, from~\cite{gls3} we obtain 
\[
 |g^G|=\begin{cases}
        \frac12q(q-1), & \textrm{if } q\equiv 3\pmod 4,\\
        \frac12q(q+1), & \textrm{if }q\equiv 1\pmod 4, \\
        q^2-1, &\textrm{if } q \textrm{ even}.
       \end{cases}
\]
Using this information, Eq.~\eqref{e: fora} and the fact that $\PSL_2(q)$ has a unique conjugacy class of involutions, it is a straightforward computation to verify the fourth and fifth column in Table~\ref{t: inv q odd} and the third and fourth column in Table~\ref{t: inv q even}.
\begin{table}
\begin{adjustbox}{angle=90}
\begin{tabular}{|c|c|c|c|c|c|c|}
\hline
$H$ & Conditions & $|\Omega|$ & $|H\cap g^G|$ & $|\Fix_\Omega(g)|$&$|\mathcal{P}(H)\cap K^G|$&$|\Fix_\Omega(K)|$\\
\hline
$[q]:(\frac{q-1}{2})$ & None & $q+1$ & $\begin{cases}
                0, & q\equiv 3(4) \\
                q, & q\equiv 1(4)
               \end{cases}$ & $\begin{cases}
                0, & q\equiv 3(4) \\
                2, & q\equiv 1(4)
               \end{cases}$  &$0$&$0$\\
$D_{q-1}$ & None & $\frac{q(q+1)}{2}$ & $\begin{cases}
                \frac{q-1}{2}, & q\equiv 3(4) \\
                \frac{q+1}{2}, & q\equiv 1(4)
               \end{cases}$ &  $\frac{q+1}{2}$&$\begin{cases}0,&q\equiv 3(8)\\ \frac{q-1}{4},&q\equiv 5(8)\\\frac{\zeta(q+1)}{8},&q\equiv 1(8)\\0,&q\equiv 7(8)\end{cases}$&$\begin{cases}0,&q\equiv 3(8)\\ 3,&q\equiv 5(8)\\3,&q\equiv 1(8)\\0,&q\equiv 7(8)\end{cases}$\\
$D_{q+1}$ & None & $\frac{q(q-1)}{2}$ & $\begin{cases}
                \frac{q+3}{2}, & q\equiv 3(4) \\
                \frac{q+1}{2}, & q\equiv 1(4)
               \end{cases}$ & $\begin{cases}
                \frac{q+3}{2}, & q\equiv 3(4) \\
                \frac{q-1}{2}, & q\equiv 1(4)
               \end{cases}$ &$\begin{cases}\frac{q+1}{4},&q\equiv 3(8)\\0,&q\equiv 5(8)\\0,&q\equiv 1(8)\\\frac{\zeta(q+1)}{8},&q\equiv 7(8)\end{cases}$&$\begin{cases}3,&q\equiv 3(8)\\ 0,&q\equiv 5(8)\\0,&q\equiv 1(8)\\3,&q\equiv 7(8)\end{cases}$\\
$\PSL_2(q_0)$ & $q=q_0^a$, $a$\textrm{ odd } & $\frac{q(q^2-1)}{q_0(q_0^2-1)}$ & $\begin{cases}
        \frac{q_0(q_0-1)}{2}, & q_0\equiv 3(4)\\
        \frac{q_0(q_0+1)}{2}, & q_0\equiv 1(4) \\
       \end{cases}$ & $\begin{cases}
        \frac{q+1}{q_0+1}, & q_0\equiv 3(4)\\
        \frac{q-1}{q_0-1}, & q_0\equiv 1(4) \\
       \end{cases}$ &$\begin{cases}\frac{q_0(q_0^2-1)}{24},&q\equiv \pm 3(8)\\\frac{\zeta q_0(q_0^2-1)}{48},&q\equiv \pm 1(8)\end{cases}$&$1$\\
$\PGL_2(q_0)$ & $q=q_0^2$, $\zeta=1$ & $\frac{\sqrt{q}(q+1)}{2}$ & $q$ & $\sqrt{q}$ &$\frac{q_0(q_0^2-1)}{24}$ or $\frac{q_0(q_0^2-1)}{8}$&$1$ or $3$\\
$\Alt(4)$ & $q=p\equiv \pm 3(8)$ & $\frac{q(q^2-1)}{24}$ & $3$ & $\begin{cases}
                \frac{q+1}{4}, & q\equiv 3(8) \\
                \frac{q-1}{4}, & q\equiv 5(8)
               \end{cases}$ &$1$&$1$\\
$\Sym(4)$ & $q=p\equiv \pm 1(8)$, $\zeta=1$ & $\frac{q(q^2-1)}{48}$ & $9$ & $\begin{cases}
                \frac{3(q+1)}{8}, & q\equiv 7(8) \\
                \frac{3(q-1)}{8}, & q\equiv 1(8)
               \end{cases}$ &$1$ or $3$&$1$ or $3$\\
$\Alt(5)$ & $\begin{array}{l}q=p, q\equiv \pm 1 (10), \textrm{ or}\\ q=p^2, p\equiv \pm 3(10) \end{array}$ & $\frac{\zeta q(q^2-1)}{120}$ & $15$ & $\begin{cases}
                \frac{\zeta(q+1)}{4}, & q\equiv 3(4) \\
                \frac{\zeta(q-1)}{4}, & q\equiv 1(4)
               \end{cases}$&$5$&$\begin{cases}\zeta,&q\equiv \pm 3(8)\\2,&q\equiv \pm 1(8)\end{cases}$ \\
               \hline
\end{tabular}
\end{adjustbox}
\caption{Fixed points of involutions in $S$ and of a Klein $4$-subgroup of $S$, for $q$ odd. The symbol $\zeta$ is defined in~\eqref{e: zeta}.}\label{t: inv q odd}
\end{table}
\begin{table}
\begin{tabular}{|c|c|c|c|}
\hline
$H$ & $|\Omega|$ & $|H\cap g^G|$ & $|\Fix_\Omega(g)|$\\
\hline
$[q]:(q-1)$ & $q+1$ & $q-1$ & $1$ \\
$D_{2(q-1)}$  & $\frac12q(q+1)$ & $q-1$ & $\frac12q$ \\
$D_{2(q+1)}$  & $\frac12q(q-1)$ & $q+1$ & $\frac12q$ \\
 $\SL_2(q_0)$ & $\frac{q(q^2-1)}{q_0(q_0^2-1)}$ & $q_0^2-1$ & $\frac{q}{q_0}$ \\
 \hline
\end{tabular}
\caption{Fixed points of involutions in $S$ for $q$ even.}\label{t: inv q even}
\end{table}

Suppose that $q\equiv \pm 3\pmod 8$ and let $K$ be a Klein $4$-subgroup of $S$. As we mentioned above, $K$ is a Sylow $2$-subgroup of $S$, all Klein $4$-subgroups of $S$ are conjugate and $\nor {\PSL_2(q)}K\cong \Alt(4)$. Therefore $|K^G|=\frac{1}{24}q(q^2-1)$. Using this and Eq.~\eqref{e: fora}, it is easy to confirm (when $q\equiv \pm 3\pmod 8$) the veracity of the sixth and seventh column  in Table~\ref{t: inv q odd}. (Note that the $\PGL_2(q_0)$ and $\Sym(4)$ rows do not apply when $q\equiv \pm3\pmod 8$.)

Suppose now that $q\equiv \pm 1\pmod 8$ and let $K$ be a Klein $4$-subgroup of $S$. In this case, there are two $S$-conjugacy classes of Klein $4$-subgroups and, regardless of the $S$-conjugacy class on which $K$ lies, we have $\nor G K\cong \Sym(4)$. In particular,
\[
 |K^G|=\frac{1}{48}\zeta q(q^2-1),
\]
where $\zeta$ is the parameter that was defined in \eqref{e: zeta}. As above, using this and Eq.~\eqref{e: fora}, it is easy to confirm (when $q\equiv \pm 1\pmod 8$) the veracity of the sixth and seventh column  in Table~\ref{t: inv q odd}. (Note that the $\Alt(4)$ row does not apply when $q\equiv \pm1\pmod 8$.)

For the proof of Theorem~\ref{t: psl2}, we also need to compute the number of fixed points of field involutions of $G$ only for certain primitive actions when $q$ is odd: this information is tabulated in Table~\ref{t: f field}. Of course, here we assume that $q$ is a square and that $G$ does contain a field automorphism of order $2$. Now observe that
\[
 |f^G|=
\begin{cases}\frac{\zeta}{2}\sqrt{q}(q+1),&\textrm{if }q \textrm{ is odd},\\
\sqrt{q}(q+1),&\textrm{if }q \textrm{ is even}.
\end{cases}
\]
From this and~\eqref{e: fora}, the veracity of Table~\ref{t: f field} follows from easy calculations (which we omit).

Note that Lemma~\ref{l: fields} means that it is convenient to assume that $G\geq \PGL_2(q)$ where this makes no difference; however for the final action in Table~\ref{t: f field}, we must assume that $G$ does {\bf not} contain $\PGL_2(q)$ since otherwise the action is not primitive. To make this clear we state the assumed value of $\zeta$ in the ``Conditions'' column in each case.

\begin{table}
\begin{tabular}{|c|c|c|c|c|}
\hline
$H$ & Conditions & $|\Omega|$ & $|\nor GH\cap f^G|$ & $|\Fix_\Omega(f)|$\\
\hline
%$[q]:(\frac{q-1}{2})$ & $-$ & $-$ & $-$ & $-$  \\
$D_{q-1}$ & $\zeta=2$ & $\frac12q(q+1)$ & $2\sqrt{q}$ & $q$ \\
$D_{q+1}$ & $\zeta=2$ & $\frac12q(q-1)$ & 0 & 0 \\
$\PSL_2(q_0)$ & $q=q_0^a$, $a$\textrm{ odd }, $\zeta=2$, & $\frac{q(q^2-1)}{q_0(q_0^2-1)}$ & $\sqrt{q_0}(q_0+1)$ & $\frac{\sqrt{q}(q-1)}{\sqrt{q_0}(q_0-1)}$  \\
%$\PGL_2(q_0)$ & $-$ & $-$ & $-$ & $-$ \\
$\Alt(5)$ & $q=p^2\equiv \pm1\pmod{10}$, $\zeta=1$ & $\frac1{120}q(q^2-1)$ & 
$10$ & $\frac{1}{6} \sqrt{q}(q-1)$  \\
               \hline
\end{tabular}
\caption{Fixed points of field automorphisms of order $2$ for selected primitive actions of $G$ with $q$ odd.}\label{t: f field}
\end{table}

%\begin{table}
%\begin{tabular}{|c|c|c|c|}
%\hline
%$H$ & $|\Omega|$ & $|\nor GH\cap g^G|$ & $|\Fix_\Omega(g)|$\\
%\hline
%$[q]:(q-1)$ & $-$ & $-$ & $-$ \\
%$D_{2(q-1)}$  & $\frac12q(q+1)$ & $2\sqrt{q}$ & $q$ \\
%$D_{2(q+1)}$  & $\frac12q(q-1)$ & $0$ & $0$ \\
% $\SL_2(q_0)$ & $\frac{q(q^2-1)}{q_0(q_0^2-1)}$ & $??$ & $??$ \\
% \hline
%\end{tabular}
%\caption{Fixed points of field automorphisms of order $2$ for selected primitive actions of $G$ with $q$ even.}\label{t: f field 2}
%\end{table}

We are now ready to prove the two lemmas that together yield Theorem~\ref{t:  psl2} for groups with socle $\PSL_2(q)$.

\begin{lem}\label{l: handy q odd}
Let $G$ be an  almost simple primitive permutation group on the set $\Omega$ with socle isomorphic to $\SLq$ with $q$ odd. If $q>9$, then $\Omega$ contains a strongly non-binary subset.
\end{lem}
\begin{proof}
Our notation here is consistent with that established above. For instance, we identify $\Omega$ with the set of $G$-conjugates of $H$. We must consider the actions corresponding to the first column of Table~\ref{t: inv q odd}. 

\noindent\textsc{Line 1: $H$ is a Borel subgroup of $S$}. In this case $G$ acts $2$-transitively on $\Omega$, but $ \alter(\Omega)\nleq G$. Thus $\Omega$ itself is a beautiful subset and hence strongly non-binary (of the type given in Example~\ref{ex: snba1}).

\noindent\textsc{Line 5: $H\cong \PGL_2(q_0)$ where $q=q_0^2$}. We regard $S$ as the projective image of those elements of $\GL_2(q)$ that have square determinant, and we may assume that $H$ consists of the projective image of those elements in $\GL_2(q_0)$ whose entries are all in $\mathbb{F}_{q_0}$.  Let $T$ be the set of diagonal elements in $S$; let $T_0:=H\cap T$, a maximal split torus in $H$; let $\alpha$ be an element of $\Fq$ that does not lie in any proper subfield of $\Fq$; and define
\[
 N_0:=\left\{\begin{pmatrix}
              1 & \alpha b \\ 0 & 1
             \end{pmatrix} \mid b\in\mathbb{F}_{q_0}\right\}.
\]
Clearly, $N_0$ is a subgroup of $G$, $T_0$ normalizes $N_0$ and $T_0\cap N_0=\{1\}$. Thus we can form the semidirect product $X:=N_0\rtimes T_0$ and we observe that $X\cap H=T_0$. Let $\Lambda$ be the orbit of $H$ under the group $X$. One obtains immediately that $\Lambda$ is a set of size $q_0$ on which $X$ acts $2$-transitively. If $q_0>5$, then $G$ does not contain a section isomorphic to $\alter(q_0)$ and we conclude that $\Lambda$ is a beautiful subset for the action of $G$ on $\Omega$. If $q_0=5$, then $q=5^2$ and we can check the result directly using \texttt{magma}~\cite{magma}.

\noindent\textsc{Line 4: $H\cong \PSL_2(q_0)$ where $q=q_0^a$ for some odd prime $a$}. We consider first the special situation where $q$ is a square. We consider $S$ as before, with $H$ the projective image of those elements in $\GL_2(q_0)$ whose entries are all in $\mathbb{F}_{q_0}$ and which have square determinant in $\mathbb{F}_{q_0}$; finally, we know that $H$ has a subgroup $H_1$ isomorphic to $\PGL_2(\sqrt{q_0})$ (since $q_0$ is a square by assumption). We take $H_1$ to be the projective image of those elements in $\GL_2(\sqrt{q_0})$ whose entries are all in $\mathbb{F}_{\sqrt{q_0}}$. Let $T$ be the set of diagonal elements in $S$; let $T_0:=H_1\cap T$, a maximal split torus in $H_1$; let $\alpha$ be an element of $\Fq$ that does not lie in any proper subfield of $\Fq$; and define
\[
 N_0:=\left\{\begin{pmatrix}
              1 & \alpha b \\ 0 & 1
             \end{pmatrix} \mid b\in\mathbb{F}_{\sqrt{q_0}}\right\}.
\]
As above, $N_0$ is a subgroup of $G$, $T_0$ normalizes $N_0$ and $T_0\cap N_0=\{1\}$. Thus we can form the semidirect product $X:=N_0\rtimes T_0$ and we observe that $X\cap H=T_0$. Let $\Lambda$ be the orbit of $H$ under the group $X$. One obtains immediately that $\Lambda$ is a set of size $\sqrt{q_0}$ on which $X$ acts $2$-transitively. If $\sqrt{q_0}>5$, then $G$ does not contain a section isomorphic to $\alter(q_0)$ and we conclude that $\Lambda$ is a beautiful subset for the action of $G$ on $\Omega$. The outstanding cases (that is, $\sqrt{q_0}\le 5$ or $q$ is not a square) will be dealt with below.

\noindent\textsc{Lines 2,3,4,6,7,8}. Here we will show that in every case we can find a strongly non-binary subset $\Lambda$ for which $G^\Lambda$ is as in Example~\ref{ex: snba2}. We let $g$ be an involution in $S$ and $h\in g^G$ with $K:=\langle g,h\rangle$ a Klein $4$-subgroup of $S$ and we let
\[
 \Lambda=\Fix(g)\cup\Fix(h)\cup\Fix(gh).
\]
Observe that $\Lambda$, $\Fix(g)$, $\Fix(h)$ and $\Fix(gh)$ are $g$-invariant and $h$-invariant. Write $\tau_1$ for the permutation induced by $g$ on $\Fix(gh)$ and $\tau_2$ for the permutation induced by $g$ on $\Fix(h)$, and observe that the supports of $\tau_1$ and $\tau_2$ are disjoint, and that $g$ induces the permutation $\tau_1\tau_2$ on $\Lambda$. Observe, furthermore, that $h$ induces the permutation $\tau_1$ on $\Fix(gh)$; now write $\tau_3$ for the involution induced by $h$ on $\Fix(g)$, and observe that the supports of $\tau_1$ and $\tau_3$ are disjoint, and that $h$ induces the permutation $\tau_1\tau_3$ on $\Lambda$.
Observe, finally, that the supports of $\tau_2$ and $\tau_3$ are disjoint and that, since $g,h$ and $gh$ are conjugate, the permutations $\tau_1,\tau_2$ and $\tau_3$ all have support of equal size.

Comparing the entries in the fifth and seventh column of Table~\ref{t: inv q odd}, we see that $|\Fix(g)|\geq |\Fix(K)|+2$. (Here we are using our assumption that $q>9$.) This implies, in particular, that $\tau_1,\tau_2$ and $\tau_3$ are non-trivial permutations of order $2$. Observe that either there exists $f\in G_\Lambda$ inducing the permutation $\tau_1$ on $\Lambda$ or else $\Lambda$ is a strongly non-binary subset of $\Omega$ (it corresponds to Example~\ref{ex: snba2}). 

Suppose that $G$ does not contain a field automorphism of order $2$, and suppose that $f\in G_\Lambda$ induces the permutation $\tau_1$ on $\Lambda$. This would imply that $G_{\Lambda}$ contained an elementary-abelian subgroup of order $8$. But, as we observed earlier, $m_2(Q)\leq 2$ for any section $Q$ in $G$, which is a contradiction. We conclude that $\Lambda$ is a strongly non-binary subset of $\Omega$.

Note that this argument disposes of Lines 6 and 7 of Table~\ref{t: inv q odd}. It also deals with one of the outstanding cases for Line 4, namely the situation where $q$ is not a square.

Suppose from here on that $G$ contains a field automorphism of order $2$. In particular $q$ is a square and $q\equiv 1\pmod 8$. Now, the previous argument implies that $G^\Lambda$ is strongly non-binary unless $G_\Lambda$ contains a field automorphism that induces the element $\tau_1$, so assume that this is the case. There are two possibilities:
 \begin{enumerate}
 \item[(a)] there is a field automorphism $f$ of order $2$ that induces the element $\tau_1$ on $\Lambda$;
 \item[(b)] there is a field automorphism $f$ of order $2$ that fixes $\Lambda$ point-wise (and some element of $G_\Lambda\cap (G\setminus \PGL_2(q))$ of order divisible by $4$ induces the element $\tau_1$ on $\Lambda$).
 \end{enumerate}

Note first that Line 3 of Table~\ref{t: inv q odd} is immediately excluded since field automorphisms of order $2$ have no fixed points in this action (see Table~\ref{t: f field}). We are left only with Lines 2 and 8, as well as Line 4 with $q_0\in\{9,25\}$. 

Assume that Case~(a) holds. Observe that, the action on $\Lambda$ gives a natural homomorphism $\langle S_{(\Lambda)},f,g,h\rangle \to \Sym(\Lambda)$ whose image is elementary abelian of order $8$, and whose kernel is $S_{(\Lambda)}$. What is more, by Lemma~\ref{l: quotients split}, $S_{(\Lambda)}$ has odd order, and we conclude that $\langle f,g,h\rangle$ is elementary abelian of order $8$.

Since $f$ centralizes $\langle g,h\rangle$ we may consider the action of $\langle g,h\rangle$ on $\Fix(f)$. Observe that if $\gamma\not\in\Fix(g)\cup \Fix(h)\cup\Fix(gh)$, then $\gamma^g\neq \gamma^h$, and so $\langle g, h\rangle$ acts semi-regularly on $\Fix(f)\setminus(\Fix(f)\cap\Lambda)$ and so
\[
 |\Fix(f)\setminus(\Fix(f)\cap\Lambda)|\equiv 0 \pmod 4.
\]
Now in this case $\Fix(f)\cap \Lambda = \Fix(g)\cup\Fix(h)$ and we conclude that
\begin{equation}\label{eq: f}
 |\Fix(f)|-2|\Fix(g)|+|\Fix(K)|\equiv 0\pmod 4.
\end{equation}
Let us consider the remaining actions, one by one.

\noindent\textsc{Line 2: $H\cong D_{q-1}$}. In this case \eqref{eq: f} implies that
 \[
  |\Fix(f)|-2|\Fix(g)|+|\Fix(K)|=q-(q+1)+3\equiv 0\pmod 4
 \]
which is a contradiction.

\noindent\textsc{Line 4: $H\cong \PSL_2(q_0)$ with $q=q_0^a$ and $a$ an odd prime}. Note first that we may assume that $q_0\in\{9,25\}$, with $p=\sqrt{q_0}$. Choose $g\in S$ to be an element of order $p$; an easy calculation using \eqref{e: fora} confirms that $g$ fixes $\frac{q}{q_0}$ points of $\Omega$. Now choose $h\in S$ to be an element of order $p$ (hence also fixing the same number of points of $\Omega$) such that $\langle g,h\rangle$ is an elementary-abelian group of order $q_0$. We require, moreover, that $\langle g,h\rangle$ fixes no points of $\Omega$: for this we just make sure that $\langle g,h\rangle$ is not conjugate to a Sylow $p$-subgroup of $H$. As usual we set $\Lambda=\Fix(g)\cup\Fix(h)\cup \Fix(gh)$. We define $\tau_1, \tau_2, \tau_3$ exactly as in the argument for \textsc{Lines 2,3,4,6,7,8}. 

Now if $f$ is an element inducing the permutation $\tau_1$, then $f$ has order divisible by $p$, and $f$ fixes at least $\frac{2q}{q_0}$ elements of $\Omega$. This implies immediately that $f\not\in S$, and we conclude that $a=p$. Now, referring to Lemma~\ref{l: fields}, we see that $f$ must be a field automorphism of degree $a=p$ and an easy calculation with \eqref{e: fora} implies that such an element fixes $\frac12p(p^2+1)$ points of $\Omega$ and so cannot induce the permutation $\tau_1$. Now, referring to Example~\ref{ex: snba2}, we conclude that $\Lambda$ is a strongly non-binary subset.
 
\noindent\textsc{Line 8: $H\cong A_5$.} In this case we assume that $G\leq {\rm P\Sigma L}_2(q)$, otherwise the action on $\Omega$ is not primitive. In particular $\zeta=1$ and \eqref{eq: f} implies that
\[
  |\Fix(f)|-2|\Fix(g)|+|\Fix(K)|=\frac{1}{6}\sqrt{q}(q-1)-\frac{1}{2}(q-1)+2\equiv 0\pmod 4.
 \]
which is a contradiction.

\smallskip

We are left with Case~(b). Note in this case that $q=p^a$ where $a$ is divisible by $4$. This immediately excludes Line 8 of the table (since $q=p^2$ here) as well as the remaining cases for Line $4$ (since here $q=9^a$ or $25^a$ where $a$ is an odd prime). Thus the only line left to consider is Line 2. But note that, for Case (b) to hold, $\Fix(f)$ must contain $\Lambda$ and so
\[
 |\Fix(f)|\geq 3|\Fix(g)|-2|\Fix(K)|.
\]
But Tables~\ref{t: inv q odd} and \ref{t: f field} then give that
\[
 q\geq \frac{3}{2}(q+1)-6.
\]
This is a contradiction for $q>9$ and we are done.
\end{proof}

\begin{lem}\label{l: handy q even}
Let $G$ be an  almost simple primitive permutation group on the set $\Omega$ with socle isomorphic to $\SLq$ with $q=2^a$. If $a>3$, then $\Omega$ contains a strongly non-binary subset.
\end{lem}
\begin{proof}
Our notation here is consistent with that established above. We must consider the actions corresponding to the first column of Table~\ref{t: inv q even}. 

\noindent\textsc{Line 1: $H$ is a Borel subgroup of $S$}. In this case $G$ acts 2-transitively on $\Omega$, but $G\not\cong \alter(\Omega)$ or $\symme(\Omega)$. Thus $\Omega$ itself is a beautiful subset (and hence strongly non-binary).

\noindent\textsc{Line 2: $H\cong D_{2(q-1)}$}. We may assume that $H$ contains $T$, the set of diagonal elements in $S$. We define
\[
 N:=\left\{\begin{pmatrix}
              1 & \alpha \\ 0 & 1
             \end{pmatrix} \mid \alpha\in\mathbb{F}_{q}\right\}.
\]
Now it is clear that $T$ normalizes $N$ and that $T\cap N=\{1\}$. Thus we can form the semidirect product $X=N\rtimes T_0$ and we observe that $X\cap H=T$.

Using the identification of $\Omega$ with the set of $G$-conjugates of $H$, we let $\Lambda$ be the orbit of $H$ under the group $X$. One obtains immediately that $\Lambda$ is a set of size $q\geq 8$ on which $X$ acts 2-transitively. Since $G$ does not contain a section isomorphic to $\alter(q)$ we conclude that $\Lambda$ is a beautiful subset for the action of $G$ on $\Omega$ and we are done.

\noindent\textsc{Line 3: $H\cong D_{2(q+1)}$}.  We proceed similarly to the case where $q$ is odd in Lemma~\ref{l: handy q odd}: let $g$ be an involution in $S$ and $h\in g^G$ with $K:=\langle g,h\rangle$ a Klein $4$-subgroup of $S$ and we let
\[
 \Lambda=\Fix(g)\cup\Fix(h)\cup\Fix(gh).
\]
Observe that $\Lambda$, $\Fix(g)$, $\Fix(h)$ and $\Fix(gh)$ are $g$-invariant and $h$-invariant. Observe, furthermore, that $\Fix(g)$, $\Fix(h)$ and $\Fix(gh)$ are all disjoint and, by Table~\ref{t: inv q even}, are of size $\frac12 q$. Write $\tau_1$ for the permutation induced by $g$ on $\Fix(gh)$, $\tau_2$ for the permutation induced by $g$ on $\Fix(h)$, and $\tau_3$ for the permutation induced by $g$ on $\Fix(gh)$.

Then $g$ induces the permutation $\tau_1\tau_2$ on $\Lambda$, while $h$ induces the permutation $\tau_1\tau_3$ on $\Lambda$. Then $\Lambda$ is a strongly non-binary subset provided there is no element $f\in G_\Lambda$ that induces the permutation $\tau_1$. Such an element must have even order and must fix at least $q$ elements of $\Omega$. Now Table~\ref{t: inv q even} implies that $f\not\in S$. On the other hand, if $f$ is a field-automorphism of order $2^c$, then it does not fix any elements of $\Omega$. We conclude that $\Lambda$ is a strongly non-binary subset and we are done.

\noindent\textsc{Line 4: $H\cong \SL_2(q_0)$ where $q=q_0^b$ for some prime $b$.} Note that, using \cite{bhr}, we can exclude the possibility that $q_0=2$. Suppose first that $q_0>4$, and take $\beta \in \mathbb{F}_q\setminus \mathbb{F}_{q_0}$. 

We may assume that $H$ consists of those elements in $S=\SL_2(q)$ whose entries are all in $\mathbb{F}_{q_0}$. Let $T$ be the set of diagonal elements in $S$; let $T_0=S\cap T$, a maximal split torus in $S$; and define
\[
 N_0:=\left\{\begin{pmatrix}
              1 & \beta \alpha \\ 0 & 1
             \end{pmatrix} \mid \alpha\in\mathbb{F}_{q_0}\right\}.
\]
Now it is clear that $T_0$ normalizes $N_0$ and that $T_0\cap N_0=\{1\}$. Thus we can form the semidirect product $X=N_0\rtimes T_0$ and we observe that $X\cap H=T_0$.

Using the identification of $\Omega$ with the set of $G$-conjugates of $H$, we let $\Lambda$ be the orbit of $H$ under the group $X$. One obtains immediately that $\Lambda$ is a set of size $q_0\geq 8$ on which $X$ acts 2-transitively. Since $G$ does not contain a section isomorphic to $\alter(q_0)$ we conclude that $\Lambda$ is a beautiful subset for the action of $G$ on $\Omega$ and we are done.

The only remaining case is when $q_0=4$. As $q=2^a$, we have $a=2b$. In this case we make use of the fact that the number of $S$-conjugacy classes of subgroups of $S$ isomorphic to a Klein 4-subgroup is $(q+2)/6$.
Since $H$ contains a unique conjugacy class of Klein 4-subgroups, there exists a Klein $4$-subgroup $K:=\langle g,h\rangle$ of $S$ with $K\nleq H^g$, for every $g\in S$, that is, $\Fix(K)=\emptyset$ for the action on cosets of $H$.

Observe that $q$ is a square. We choose $K$ so that, not only does it not lie in a conjugate of $H$, it also doesn't lie in a conjugate of $\SL_2(\sqrt{q})=\SL_2(2^b)$, the centralizer of a field automorphism of order $2$. Define
\[
 \Lambda=\Fix(g)\cup\Fix(h)\cup\Fix(gh).
\]
Observe that $g$ acts on $\Lambda$, and on $\Fix(h)$, and on $\Fix(gh)$. Write $\tau_1$ for the involution induced by $g$ on $\Fix(gh)$ and $\tau_2$ for the permutation induced by $g$ on $\Fix(h)$, and observe that the supports of $\tau_1$ and $\tau_2$ are disjoint, and that $g$ induces the permutation $\tau_1\tau_2$ on $\Lambda$.

Exactly as in the case when $H=D_{2(q-1)}$, $\Lambda$ is either strongly non-binary (and we are done), or else there exists $f\in G^\Lambda$ such that $f$ induces the permutation $\tau_1$ on $\Lambda$. Suppose that this latter possibility occurs, and observe that $\Fix(f)$ contains $\Fix(g)\cup\Fix(h)$ and so $|\Fix(f)|\geq \frac{q}{2}$. If $f\in S$, then $f$ is conjugate to $g$ and $|\Fix(g)|=\frac{q}{4}$, so we have a contradiction.

Suppose that $f\not\in S$. The subgroup structure of $\SL_2(q)$ implies that if a subgroup $X$ is normalized by a Klein $4$-group, then $X$ is elementary abelian of even order. Thus $S_{(\Lambda)}$ is elementary abelian of even order. But if $S_{(\Lambda)}$ is non-trivial, then an involution fixes at least $\frac{3q}{4}$ points which is a contradiction. Thus $S_{(\Lambda)}$ is trivial. 

Now, note that since $q=4^b$, where $b$ is prime, either $q=16$, or else we may assume that $f$ is a field automorphism of order $2$. Thus $\langle K, f\rangle$ is elementary-abelian. But, since $K$ does not lie in a conjugate of $\SL_2(\sqrt{q})$ we have a contradiction here. In the case $q=16$, a moment's thought shows that either $f$ is a field automorphism of order $2$, or else there is a field automorphism of order $2$ that fixes $\Lambda$ point-wise. Either way one concludes that there is a field automorphism $f_1$ such that $\langle K, f_1\rangle$ is elementary-abelian and, again, we have a contradiction. Thus in all cases we have a strongly non-binary subset of the type given in Example~\ref{ex: snba2} and we are done.
\end{proof}

We remark again that Theorem~\ref{t: psl2} for groups with socle $\PSL_2(q)$ is an immediate consequence of Lemmas~\ref{l: forbidden}, \ref{l: handy q odd} and \ref{l: handy q even}.

\section{Groups with socle isomorphic to \texorpdfstring{$^{2}B_2(q)$}{2B2(q)}}\label{s: suzuki}

In this section we prove Theorem~\ref{t: psl2} for groups with socle $\suzuki$. This theorem follows, {\it \`a la} the other main results, from Lemma~\ref{l: suzuki} combined with Lemma~\ref{l: forbidden}. In what follows $G$ is an almost simple group with socle $S\cong\suzuki$, where $q=2^a$ and $a$ is an odd integer with $a\geq 3$. We write $r:=2^{\frac{a+1}{2}}$ and define $\theta$ to be the following field automorphism of $\Fq$:
\[
 \theta: \Fq\to \Fq, \,\,\, x \mapsto x^r.
\]

We need some basic facts, all of which can be found in \cite{suzuki}. First, ${\rm Out}(S)$ is a cyclic group of odd order $a$. Second, $G$ contains a single conjugacy class of involutions; writing $g$ for one of these involutions we note that 
\[
 |g^G|=(q^2+1)(q-1).
\]
Third, the maximal subgroups of $S$ fall into three families: Borel subgroups, normalizers of maximal tori, and subfield subgroups. For the second of these families, we need some fixed point calculations, and these are given in Table~\ref{t: suz} (making use of \eqref{e: fora}).  Each line of this table corresponds to a conjugacy class of maximal tori in $S$; we write $H$ for a maximal subgroup of $S$ and $\Omega$ for the set of right cosets of $H$ in $S$; in the final column we write $K$ for a Klein $4$-subgroup of $S$.
\begin{table}
\begin{tabular}{|c|c|c|c|c|c|}
\hline
$H$ & $|\Omega|$ & $|H\cap g^G|$ & $|\Fix_\Omega(g)|$ & $|\mathcal{P}(H)\cap K^S|$&$|\Fix_\Omega(K)|$\\
\hline
$D_{2(q-1)}$  & $\frac12q^2(q^2+1)$ & $q-1$ & $\frac12q^2$&$0$ & $0$\\
$(q+r+1)\rtimes 4$  & $\frac14q^2(q-1)(q-r+1)$ & $q+r+1$ & $\frac14q^2$ &$0$ & $0$\\
$(q-r+1)\rtimes 4$  & $\frac14q^2(q-1)(q+r+1)$ & $q-r+1$ & $\frac14q^2$ &$0$ &$0$\\
 \hline
\end{tabular}
\caption{Fixed points of involutions and Klein $4$-subgroups for selected primitive actions of almost simple Suzuki groups.}\label{t: suz}
\end{table}

\begin{lem}\label{l: suzuki}
  Let $G$ be an almost simple primitive permutation group on the set $\Omega$ with socle $S\cong\suzuki$. Then $\Omega$ contains a strongly non-binary subset.
\end{lem}

\begin{proof}
Note that $|S|$ is not divisible by $3$ and hence $G$ does not contain a section isomorphic to an alternating group $\Alt(n)$ with $n\geq 3$.
Referring to \cite{suzuki}, we see that a maximal subgroup of $G$ is necessarily the normalizer in $G$ of a maximal subgroup $H$ of $S$. Thus we can identify $\Omega$ with the set of right cosets of $H$ in $S$. We split into three families, as per the discussion above.

 First, if $H$ is a Borel subgroup, then the action of $G$ on $\Omega$ is $2$-transitive and, since $G$ contains no alternating sections, we obtain immediately that $\Omega$ itself is a beautiful subset.
 
 Second, if $H$ is the normalizer in $S$ of a maximal torus, then we set $K$ to be a Klein 4-subgroup of $S$, and we let $g,h$ be distinct involutions in $K$. Referring to Table~\ref{t: suz}, we see that $g$ and $h$ fix at least $16$ points of $\Omega$, while $K$ fixes none. We set $\lambda_3$ to be one of the fixed points of $g$ and write $\lambda_4$ for the point $\lambda_3^h$. Similarly $\lambda_5\in\Fix(h)$ and $\lambda_6=\lambda_5^g$. Finally pick $\lambda_1\in\Fix(gh)$ and let $\lambda_2=\lambda_1^g$; observe that $\lambda_2=\lambda_1^h$. Now let $\Lambda=\{\lambda_1,\lambda_2,\lambda_3,\lambda_4,\lambda_5,\lambda_6\}$ and observe that $K$ acts on this set with the element $g$ inducing the permutation $(\lambda_1,\lambda_2)(\lambda_5,\lambda_6)$ while the element $h$ induces the permutation $(\lambda_1,\lambda_2)(\lambda_3,\lambda_4)$. Suppose that $f\in G_\Lambda$ induces the permutation $(\lambda_1,\lambda_2)$ on $\Lambda$. This would imply that $f$ and $g$ fix the point $\lambda_3$ and so the stabilizer of $\lambda_3$ must contain a section isomorphic to a Klein 4-subgroup. This is impossible: the Sylow $2$-subgroups of $H$ are cyclic of order $2$ or $4$ and, since $|{\rm Out}(S)|$ is odd, this is true of the stabilizer in $G$ of $\lambda_3$. Therefore $\Lambda$ is a strongly non-binary subset of $\Omega$: it corresponds to Example~\ref{ex: snba2}.
 
 Third, suppose that $H$ is a subfield subgroup of $S$. It is convenient to take $S$ to be the set of $4\times 4$ matrices over $\Fq$ described on \cite[p.133]{suzuki}; then we take $H$ to be the subgroup of $S$ consisting of matrices with entries over $\mathbb{F}_{q_0}$ with $q=q_0^b$ for some prime $b$, and $q_0>2$. The following set forms a Sylow $2$-subgroup of $S$:
 \[
  P_2(q):= \left\{\begin{pmatrix}
          1 & 0 & 0 & 0 \\ \alpha & 1 & 0 & 0 \\ \alpha^{1+\theta}+\beta & \alpha^\theta & 1 & 0 \\ \alpha^{2+\theta}+\alpha\beta+\beta^\theta & \beta&\alpha & 1
         \end{pmatrix} \mid
   \alpha, \beta \in \Fq \right\}.
 \]
The subgroup $P_2(q)$ is normalized by the following subgroup  of $S$,
\[
 K(q):= \left\{\begin{pmatrix}
          \zeta_1 & 0 & 0 & 0 \\ 0 & \zeta_2 & 0 & 0 \\ 0 & 0 & \zeta_3 & 0 \\ 0 & 0 & 0 & \zeta_4
         \end{pmatrix} \mid
   \exists\kappa \in \Fq\setminus\{0\}, \zeta_1^\theta = \kappa^{1+\theta}, \zeta_2^\theta=\kappa, \zeta_3=\zeta_2^{-1}, \zeta_4=\zeta_1^{-1} \right\}.
\]
The group $P_2(q)\rtimes K(q)$ is a maximal Borel subgroup of $S$, while $P_2(q_0)\rtimes K(q_0)$ is a maximal Borel subgroup of $H$. Observe that the center $\Zent {P_2(q)}$ of $P_2(q)$ consists of those matrices for which $\alpha=0$.

Let $\zeta\in\Fq\setminus\mathbb{F}_{q_0}$ and consider the group
\[
 ZP_2(\zeta,q_0):= \left\{\begin{pmatrix}
          1 & 0 & 0 & 0 \\ 0 & 1 & 0 & 0 \\ \zeta\beta & 0 & 1 & 0 \\ (\zeta\beta)^\theta & \zeta\beta & 0 & 1
         \end{pmatrix} \mid
   \beta \in \mathbb{F}_{q_0} \right\}.
\]
Observe that $K(q_0)$ normalizes $ZP_2(\zeta,q_0)$, that $K(q_0)<H$, that $ZP_2(\zeta,q_0)\cap H=\{1\}$ and that $K(q_0)$ acts fixed-point-freely on $ZP_2(\zeta,q_0)$. Let $X:=ZP_2(\zeta,q_0)\rtimes K(q_0)$; identifying $\Omega$ with the set of $G$-conjugates of $H$, we let $\Lambda$ be the orbit of $H$ under the group $X$. One obtains immediately that $\Lambda$ is a set of size $q_0\geq 8$ on which $X$ acts 2-transitively. 
The absence of alternating sections implies that $\Lambda$ is a beautiful subset. 
\end{proof}

\section{Groups with socle isomorphic to \texorpdfstring{$^{2}G_2(q)$}{2G2(q)}}\label{s: ree}

In this section we prove Theorem~\ref{t: psl2} for groups with socle $\ree$. This theorem follows, {\it \`a la} the other main results, from Lemma~\ref{l: ree} combined with Lemma~\ref{l: forbidden}. In what follows $G$ is an almost simple group with socle $S\cong\ree$, where $q=3^a$ and $a$ is an odd integer with $a\geq 3$. We write $r:=3^{\frac{a+1}{2}}$ and define $\theta$ to be the following field automorphism of $\Fq$:
\[
 \theta: \Fq\to \Fq, \,\,\, x \mapsto x^r.
\]

We need some basic facts, all of which can be found in \cite{kleidman}. First, ${\rm Out}(S)$ is a cyclic group of odd order $a$. Second, $G$ contains a single conjugacy class of involutions; writing $g$ for one of these involutions we note that 
\[
 |g^G|=q^2(q^2-q+1).
\]
Third, the order of $S$ is not divisible by $5$, and so $G$ does not contain a section isomorphic to $\alter(n)$ with $n\geq 5$.
Fourth, the maximal subgroups of $G$ fall into four families: Borel subgroups, normalizers of maximal tori, involution centralizers, and subfield subgroups. For all but the first of these families, we need some fixed point calculations, and these are given in Table~\ref{t: ree} (making use of \eqref{e: fora}).  In each line of this table we write $H$ for a maximal subgroup of $S$ and $\Omega$ for the set of right cosets of $H$ in $S$; in the final column we write $K$ for a Klein $4$-subgroup of $S$.
\begin{table}
\begin{tabular}{|c|c|c|c|c|c|}
\hline
$H$ & $|\Omega|$ & $|H\cap g^G|$ & $|\Fix_\Omega(g)|$ & $|K^G\cap \mathcal{P}(H)|$ & $|\Fix_\Omega(K)|$\\
\hline
$(2^2\times D_{\frac{q+1}{2}})\rtimes 3$  & $\frac{q^3(q^2-q+1)(q-1)}{6}$ & $q+4$ & $\frac{q(q-1)(q+4)}{6}$ & $\frac{3q+5}{2}$ & $\frac{3q+5}{2}$\\
$(q+r+1)\rtimes 6$  & $\frac{q^3(q^2-1)(q-r+1)}{6}$ & $q+r+1$ & $\frac{q(q^2-1)}{6}$ & $0$ & $0$ \\
$(q-r+1)\rtimes 6$  & $\frac{q^3(q^2-1)(q+r+1)}{6}$ & $q-r+1$ & $\frac{q(q^2-1)}{6}$ & $0$ & $0$ \\
$2\times \PSL_2(q)$  & $q^2(q^2-q+1)$ & $q^2-q+1$ & $q^2-q+1$ & $\frac{(q+4)q(q-1)}{6}$ & $q+4$\\
${^2G_2}(q_0)$ & $\frac{q^3(q^3+1)(q-1)}{q_0^3(q_0^3+1)(q_0-1)}$ & $q_0^2(q_0^2-q_0+1)$ & $\frac{q(q^2-1)}{q_0(q_0^2-1)}$ & $\frac16q_0^3(q_0^2-q_0+1)(q_0-1)$ & $\frac{q+1}{q_0+1}$ \\
\hline
\end{tabular}
\caption{Fixed points of involutions and Klein $4$-subgroups for selected primitive actions of almost simple Ree groups.}\label{t: ree}
\end{table}

The calculations given in Table~\ref{t: ree} make use of the fact there is a unique class of involutions and a unique class of Klein $4$-subgroups in $S$; their normalizers are maximal subgroups. In particular the normalizer of a Klein $4$-subgroup in $S$ is the group $H$ in the first line of Table~\ref{t: ree}; combined with the fact that a Sylow $2$-subgroup of $S$ is elementary abelian of order $8$, we are able to complete the final entry in that first row. The other entries in the table follow from easy calculations.

\begin{lem}\label{l: ree}
Let $G$ be an almost simple primitive permutation group on the set $\Omega$ with socle $S\cong\ree$. Then $\Omega$ contains a strongly non-binary subset.
\end{lem}
\begin{proof}
 Referring to \cite{kleidman}, we see that a maximal subgroup of $G$ is necessarily the normalizer in $G$ of a maximal subgroup $H$ of $S$. Thus we can identify $\Omega$ with the set of right cosets of $H$ in $S$. We split into two cases.

 First, if $H$ is a Borel subgroup, then the action of $G$ on $\Omega$ is $2$-transitive and, since $G$ contains no sections isomorphic to $\alter(n)$ with $n\geq 5$, we obtain immediately that $\Omega$ itself is a beautiful subset.
 
Second, if $H$ is not a Borel subgroup, then we set $K$ to be a Klein 4-subgroup of $S$, we let $g,h$ be distinct involutions in $K$, and we let
\[
 \Lambda=\Fix(g)\cup\Fix(h)\cup\Fix(gh).
\]
Observe that $\Lambda$, $\Fix(g)$, $\Fix(h)$ and $\Fix(gh)$ are $g$-invariant and $h$-invariant. Write $\tau_1$ for the involution induced by $g$ on $\Fix(gh)$ and $\tau_2$ for the permutation induced by $g$ on $\Fix(h)$, and observe that the supports of $\tau_1$ and $\tau_2$ are disjoint, and that $g$ induces the permutation $\tau_1\tau_2$ on $\Lambda$. Observe, furthermore, that $h$ induces the permutation $\tau_1$ on $\Fix(gh)$; now write $\tau_3$ for the involution induced by $h$ on $\Fix(g)$, and observe that the supports of $\tau_1$ and $\tau_3$ are disjoint, and that $h$ induces the permutation $\tau_1\tau_3$ on $\Lambda$.

Observe, finally, that the supports of $\tau_2$ and $\tau_3$ are disjoint. Now, suppose that $f\in G_\Lambda$ induces the permutation $\tau_1$ on $\Lambda$. This would imply that $f$ fixes more points than $g$. Since $f$ has even order and all involutions in $G$ are conjugate, some odd power of $f$ is a conjugate of $g$, which is a contradiction. Thus $\Lambda$ is a strongly non-binary subset of $\Omega$ (it corresponds to Example~\ref{ex: snba2}). 
\end{proof}

\section{Groups with socle isomorphic to \texorpdfstring{$\PSU_3(q)$}{PSU(3,q)}}\label{s: psu}

In this section we prove Theorem~\ref{t: psl2} for groups with socle $\PSU_3(q)$. Strictly speaking, this theorem does not follow {\it \`a la} the other main results. Firstly, we do not prove the existence of beautiful subsets or of strongly non-binary subsets: we simply prove that the primitive groups under consideration are not binary. Second, for some primitive actions we make use of computer aided computations. The basic ideas for these computations are inspired from a deeper analysis in~\cite{DV_NG_PS}, where Conjecture~\ref{conj: cherlin} is proved for most almost simple groups with socle a sporadic simple group.

The following lemmas are taken from~\cite{DV_NG_PS} and are stated in a form tailored to our needs in this paper.
\begin{lem}\label{l: again0}Let $G$ be a transitive group on a set $\Omega$, let $\alpha$ be a point of $\Omega$ and let $\Lambda\subseteq \Omega$ be a $G_\alpha$-orbit. If  $G$ is binary, then $G_\alpha^\Lambda$ is binary. 
\end{lem}
\begin{proof}Assume that $G$ is binary. Let $\ell\in\mathbb{N}$ and let $I:=(\lambda_1,\lambda_2,\ldots,\lambda_\ell)$ and $J:=(\lambda_1',\lambda_2',\ldots,\lambda_\ell')$ be two tuples in $\Lambda^\ell$ which are $2$-subtuple complete for the action of $G_\alpha$ on $\Lambda$. Clearly, $I_0:=(\alpha,\lambda_1,\lambda_2,\ldots,\lambda_\ell)$ and $J_0:=(\alpha,\lambda_1',\lambda_2',\ldots,\lambda'_\ell)$ are $2$-subtuple complete for the action of $G$ on $\Omega$; as $G$ is binary, $I_0$ and $J_0$ are in the same $G$-orbit; hence $I$ and $J$ are in the same $G_\alpha$-orbit. From this we deduce that $G_\alpha^\Lambda$ is binary. 
\end{proof}

We caution the reader that in the next lemma when we write $\Lambda$ we \emph{are not} referring to a subset of $\Omega$ -- here the set $\Lambda$ is allowed to be any set whatsoever that satisfies the listed suppositions.

\begin{lem}\label{l: again}
Let $G$ be a primitive group on a set $\Omega$, let $\alpha$ be a point of $\Omega$, let $M$ be the stabilizer of $\alpha$ in $G$ and let $d$ be an integer. Suppose $M\ne 1$ and, for each transitive action of $M$ on a set $\Lambda$ satisfying:
\begin{enumerate}
\item $|\Lambda|>1$, and 
\item every composition factor of $M$ is isomorphic to some section of $M^\Lambda$, and
\item $M$ is binary in its action on $\Lambda$,
\end{enumerate}
we have that $d$ divides $|\Lambda|$. Then either $d$ divides $|\Omega|-1$ or $G$ is not binary.
\end{lem}
\begin{proof}
Suppose that $G$ is binary. Since $\{\beta\in\Omega\mid \beta^m=\beta,\forall m\in M\}$ is a block of imprimitivity for $G$ and since $G$ is primitive, we obtain that either $M$ fixes each point of $\Omega$ or $\alpha$ is the only point fixed by $M$. The former possibility is excluded because $M\neq 1$ by hypothesis. Therefore $\alpha$ is the only point fixed by $M$. Let $\Lambda\subseteq\Omega\setminus\{\alpha\}$ be an $M$-orbit.  Thus $|\Lambda|>1$ and (1) holds. Since $G$ is a primitive group on $\Omega$, from~\cite[Theorem~3.2C]{dixon_mortimer}, we obtain that every composition factor of $M$ is isomorphic to some section of $M^\Lambda$ and hence (2) holds. From Lemma~\ref{l: again0}, the action of $M$ on $\Lambda$ is binary and hence (3) also holds. Therefore, $d$ divides $|\Lambda|$ and hence each orbit of $M$ on $\Omega\setminus\{\alpha\}$ has cardinality divisible by $d$. Thus $|\Omega|-1$ is divisible by $d$.  
\end{proof}

\begin{proof}[Proof of Theorem~$\ref{t: psl2}$ for almost simple groups with socle $\PSU_3(q)$.]

Let $G$ be an almost simple primitive group on the set $\Omega$ with socle $S$ isomorphic to $\PSU_3(q)$. Observe that $q\ge 3$ because $\PSU_3(2)$ is soluble. When $q\le 9$, we can check directly with \texttt{magma} the veracity of our statement by constructing all the primitive actions under consideration and checking one-by-one that none is binary (in each case we are able to exhibit a non-binary witness). For the rest of the proof we assume that $q>9$, that is, $q\ge 11$: among other things, this will allow us to exclude some ``novelties'' in dealing with the maximal subgroups of $G$. Moreover,  we let $V:=\mathbb{F}_{q^2}^3$ be the natural $3$-dimensional Hermitian space over the field $\mathbb{F}_{q^2}$ of cardinality $q^2$ for the appropriate covering group of $G$.

Let $\alpha\in \Omega$ and let $M:=G_\alpha$ be the stabilizer in $G$ of the point $\alpha$. We subdivide the proof according to the structure of $M$ as described in~\cite[Section~8, Tables~8.5,~8.6]{bhr}. In this proof we use~\cite{bhr} as a crib.

\smallskip

\noindent\textsc{The group $M$ is in the Aschbacher class $\mathcal{C}_1$. }This case is completely settled in~\cite{gs_binary}, where the authors have proved Cherlin's conjecture for almost simple classical groups acting on the cosets of a maximal subgroup in the Aschbacher class $\mathcal{C}_1$.

\smallskip

\noindent\textsc{The group $M$ is in the Aschbacher class $\mathcal{C}_2$.}  From~\cite{bhr}, we get that the action of $G$ on $\Omega$ is permutation equivalent to the natural action of $G$ on 
\[
\{\{V_1,V_2,V_3\}\mid \dim_{\mathbb{F}_{q^2}}(V_1)=
\dim_{\mathbb{F}_{q^2}}(V_2)=\dim_{\mathbb{F}_{q^2}}(V_3)=1,
V=V_1\perp V_2\perp V_3, V_1,V_2,V_3 \textrm{ non isotropic}\}. 
\]
 Therefore we identify $\Omega$ with the latter set. Let $e_1,e_2,e_3$ be the canonical basis of $V$ and, replacing $G$ by a suitable conjugate, we may assume that the matrix associated to the Hermitian form on $V$ with respect to the basis $e_1,e_2,e_3$ is the identity matrix. Thus $\omega_0:=\{\langle e_1\rangle,\langle e_2\rangle,\langle e_3\rangle\}\in\Omega$. Consider $\Omega_0:=\{\{V_1,V_2,V_3\}\in \Omega\mid V_1=\langle e_1\rangle\}$. Clearly, $G_{\Omega_0}=G_{\langle e_1\rangle}$,  $G_{\langle e_1\rangle}$ is a classical group, $G_{\Omega_0}/\Zent {G_{\Omega_0}}$ is almost simple with socle isomorphic to $\PSL_2(q)$ (here we are using $q>3$), and the action of $G_{\Omega_0}$ on $\Omega_0$ is permutation equivalent to the action of $G_{\langle e_1\rangle}$ on $\Omega_0':=\{\{W_1,W_2\}\mid \dim(W_1)=\dim(W_2), \langle e_1\rangle^\perp=W_1\perp W_2, W_1,W_2 \textrm{ non degenerate}\}$. Therefore $G^{\Omega_0}$ is an almost simple primitive group with socle isomorphic to $\PSL_2(q)$ and having degree $|\Omega_0|=q(q-1)/2$. Applying Theorem~\ref{t: psl2} to $G^{\Omega_0}$, we obtain that $G^{\Omega_0}$ is not binary and hence there exist two $\ell$-tuples $(\{W_{1,1},W_{1,2}\},\ldots,\{W_{\ell,1},W_{\ell,2}\})$ and $(\{W'_{1,1},W'_{1,2}\},\ldots,\{W'_{\ell,1},W'_{\ell,2}\})$ in $\Omega_0^\ell$ which are $2$-subtuple complete for the action of $G_{\Omega_0}$ but not in the same $G_{\Omega_0}$-orbit. By construction the two $\ell$-tuples 
 \begin{align*}
I& :=(\{\langle e_1\rangle,W_{1,1},W_{1,2}\},\{\langle e_1\rangle,W_{2,1},W_{2,2}\},\ldots,\{\langle e_1\rangle,W_{\ell,1},W_{\ell,2}\}), \\
J&:=(\{\langle e_1\rangle,W'_{1,1},W'_{1,2}\},\{\langle e_1\rangle,W'_{2,1},W'_{2,2}\},\ldots,\{\langle e_1\rangle,W'_{\ell,1},W'_{\ell,2}\})  
 \end{align*}
 are in $\Omega^\ell$ and are $2$-subtuple complete. Moreover, a moment's thought yields that $I$ and $J$  are not in the same $G$-orbit. Thus $G$ is not binary.

\smallskip
\noindent\textsc{The group $M$ is in the Aschbacher class $\mathcal{C}_3$. }Here $M$ is the normalizer in $G$ of a maximal non-split torus $T$ of $S$ of order $(q^2-q+1)/\gcd(q+1,3)$. From~\cite{bhr}, we infer that $\nor S T$ is a split extension of $T$ by a cyclic group $\langle x\rangle$ of order $3$ (arising from an element of order $3$ in the Weyl group of $S$), thus $M=C\rtimes K$, with $C=\langle c\rangle$ cyclic such that $C\cap S=T$ and with $K$ abelian. (The group $K$ is the direct product of a cyclic group of order $3$ and a cyclic group of order $|G:G\cap\PGU_3(q)|$.) An inspection of the maximal subgroups of $\PSU_3(q)$ reveals that there exists $g\in \nor S K\setminus M$. Set $\beta:=\alpha^g$. Since $g\notin M$, we get $\beta\neq \alpha$ and, since $g\in \nor S K$, we get $G_\alpha\cap G_\beta=M\cap M^g\ge K$. Therefore $G_\alpha\cap G_\beta=C'\rtimes K$, for some cyclic subgroup $C'$ of $C$.

Set $\Lambda:=\beta^M$. Now, the action induced by $M$ on the $M$-orbit $\Lambda$ is permutation isomorphic to the action of $M=C\rtimes K$ on the right cosets of $M\cap M^g=C'\rtimes K$. We use the ``bar'' notation and denote by $\bar{M}$ the group $M^\Lambda$. Thus $\bar{M}=\langle\bar{c}\rangle\rtimes \bar{K}$ and the action of $\bar{M}$ on $\Lambda$ is permutation isomorphic to the natural action of $\langle\bar{c}\rangle\rtimes\bar{K}$ on $\langle\bar{c}\rangle$: with $\langle\bar{c}\rangle$ acting on $\langle\bar{c}\rangle$ via its regular representation and with $\bar{K}\cong K$ acting on $\langle\bar{c}\rangle$ via conjugation. Now, $\bar{c}^{\bar{x}}=\bar{c}^\kappa$, for some $\kappa\in\mathbb{Z}$ with $\kappa^3\equiv 1\pmod {|\bar{c}|}$ and $\kappa\not\equiv 1\pmod {|\bar{c}|}$. Consider the two triples $I:=(1,\bar{c},\bar{c}^{1+\kappa^2})$ and $J:=(1,\bar{c},\bar{c}^{1+\kappa})$. Now $(1,\bar{c})^{id_{\bar{M}}}=(1,\bar{c})$, $(1,\bar{c}^{1+\kappa^2})^{\bar{x}}=(1,\bar{c}^{\kappa+\kappa^3})=(1,\bar{c}^{\kappa+1})$ and $$(\bar{c},\bar{c}^{1+\kappa^2})^{\bar{c}^{-1}\bar{x}^2\bar{c}}=(\bar{c}^{\bar{c}^{-1}\bar{x}^2\bar{c}},(\bar{c}^{1+\kappa^2})^{\bar{c}^{-1}\bar{x}^2\bar{c}})=(\bar{c},\bar{c}^{\kappa^4+1})=(\bar{c},\bar{c}^{\kappa+1}).$$
Thus $I$ and $J$ are $2$-subtuple complete for the action of $\bar{M}$ on $\langle\bar{c}\rangle$. Observe that $I$ and $J$ are not in the same $\bar{M}$-orbit because the only element of $\bar{M}$ fixing $1$ and the generator $\bar{c}$ of $\langle\bar{c}\rangle$ is the identity, but $\bar{c}^{1+\kappa^2}\ne \bar{c}^{1+\kappa}$ because $\kappa\not\equiv 1\pmod {|\bar{c}|}$. Therefore $\bar{M}$ is not binary. From Lemma~\ref{l: again0}, we deduce that $G$ is not binary.

\smallskip 
\noindent\textsc{The group $M$ is in the Aschbacher class $\mathcal{C}_5$.} Let $H$ be the stabilizer in $M$ of a non-isotropic $1$-dimensional subspace $\langle v\rangle$ of $V$ and let $K$ be the stabilizer of $\langle v\rangle$ in $G$. Thus $K$ is a maximal subgroup of $G$ in the Aschbacher class $\mathcal{C}_1$; moreover, using~\cite{bhr} and $q>8$, we see that there exists $g\in \Zent K\setminus M$. Set $\beta:=\alpha^g$. Since $g\notin M$, we get $\beta\neq \alpha$ and, since $g\in \Zent K$, we get $G_\alpha\cap G_\beta=M\cap M^g\ge M\cap K=H$. Since $H$ is maximal in $M$, we obtain $G_\alpha\cap G_\beta=H$ and hence the action induced by $G_\alpha=M$ on the $G_\alpha$-orbit $\beta^{G_\alpha}$ is permutation isomorphic to the action of $M$ on the right cosets of $H$. By construction, this latter action is the natural action of the classical group $M$ on the cosets of a maximal subgroup in its $\mathcal{C}_1$-Aschbacher class. From~\cite[Theorem~B]{gs_binary}, this action is not binary. Therefore, $G$ is not binary by Lemma~\ref{l: again0}. 

\smallskip

\noindent\textsc{The group $M$ is in the Aschbacher class $\mathcal{C}_6$ or in the Aschbacher class $\mathcal{S}$.} Now the isomorphism class of $M$ is explicitly given in~\cite{bhr}. Since $|M|$ is very small (actually $|M|\le 720$), with the invaluable help of the computer algebra system \texttt{magma}, we compute all the transitive $M$-sets and we select the $M$-sets $\Lambda$ with $|\Lambda|>1$, with every composition factor of $M$ isomorphic to some section of $M^\Lambda$ and with  $M^\Lambda$ binary. In all cases, we see that $|\Lambda|$ is even. Therefore, applying Lemma~\ref{l: again}, we obtain that either $G$ is not binary or $|\Omega|-1$ is even. In the 
latter case, $|\Omega|$ is odd and hence $M$ contains a Sylow $2$-subgroup of $G$. From~\cite{bhr}, we see that a Sylow $2$-subgroup of $M\cap S$ has size $8$, but this is a contradiction because $|\PSU_3(q)|$ is always divisible by $16$.

\end{proof}
%One naturally wonders about other low-rank groups; one candidate looms particularly large: a proof of Cherlin's conjecture for groups with socle $\reeb$ would complete the proof for all of the finite groups of Lie type that are not Steinberg or Chevalley groups. One would hope to exploit the subgroup structure given in \cite{malle}, although this structure is rather complicated and there may be significant technical obstructions.

\end{document}